\documentclass[11pt]{article}
\usepackage[letterpaper,hmargin=1in,vmargin=1.25in]{geometry}
\usepackage{amssymb, amsmath, amsthm, hyperref, graphicx}
\pdfoutput=1

\def\L{{\sf L}}
\def\R {{\sf R}}
\def\NN{{\mathbb N}}
\def\onemodsixperm{
  \left(\sigma_\R^{-1} \sigma_\L\right)^3 \left(\sigma_\R^2 \sigma_\L^2 \right)^{(k-1)/2} \sigma_\L^2 
}

\newtheorem{theorem}{Theorem}[section]
\newtheorem{lemma}[theorem]{Lemma}

\begin{document}

\title{Groups of Rotating Squares 
\thanks{Research supported by 
NSF grant DMS-9322070}
\author {
Ravi Montenegro \thanks{Department of Mathematical Sciences, University of Massachusetts Lowell, Lowell, MA 02474; email: \texttt{ravi\_montenegro@uml.edu}}
\and David A. Huckaby  \thanks{Department of Mathematics and Computer Science, Angelo State University, San Angelo, TX 76909; email: \texttt{david.huckaby@angelo.edu}}
\and Elaine White Harmon \thanks{Formerly at Department of Mathematics, McMurry University, Abilene, Texas 79697}
}
}

\date{}

\maketitle
 
\begin{abstract}
\noindent
This paper discusses the permutations that are generated by rotating $k \times k$ blocks of squares in a union of overlapping $k \times (k+1)$ rectangles.
It is found that the single-rotation parity constraints effectively determine the group of accessible permutations.
If there are $n$ squares, and the space is partitioned as a checkerboard with $m$ squares shaded and $n-m$ squares unshaded, then the four possible cases are $A_n$, $S_n$, $A_m \times A_{n-m}$, and the subgroup of all even permutations in $S_m \times S_{n-m}$, with exceptions when $k = 2$ and $k = 3$.
\end{abstract}

\section{Introduction}

Many games with a mathematical flavor involve moving blocks or balls according to some simple rotational or translational rule in an attempt to put them into some specified pattern. Generally not all permutations of the blocks are possible, and potential moves overlap at only a few elements.
For instance, with the Rubik's Cube the arrangement of the center squares on the faces is constrained and two rotations can affect at most $3$ pieces in common,
while in Hungarian Ring puzzles only $2$ of the marbles are shared by both rotations.
The $15$--puzzle involves very small moves: tiles numbered $1$ through $15$ are placed in a $4\times 4$ grid, and the only moves involve sliding a tile into the empty spot, so that only two tiles are affected by each move.
Given that in each case small moves and limited overlap are intended primarily to produce games that are more easily playable by humans, it is natural to ask what would happen in the opposite extreme.
In particular, would a rotational block-style puzzle with only a few, highly overlapping, rotations result in a large or a small number of accessible permutations of the blocks?

An extreme case of this is to consider rotating $k\times k$ blocks of square tiles in a $k\times (k+1)$ rectangle, or an overlapping union of such rectangles.
For instance, Figure \ref{fig:floor} illustrates a tile arrangement of three $k\times (k+1)$ rectangles for $k=6$, along with three potential rotations of $6 \times 6$ blocks of tiles.  (There are three other potential rotations in this particular arrangement.)  We find that in general---somewhat surprisingly---nearly all permutations of the tiles are possible, despite the minimal overlap shared by the $k\times (k+1)$ regions.

\begin{figure}[!ht]
\begin{center}
\includegraphics[height=2in]{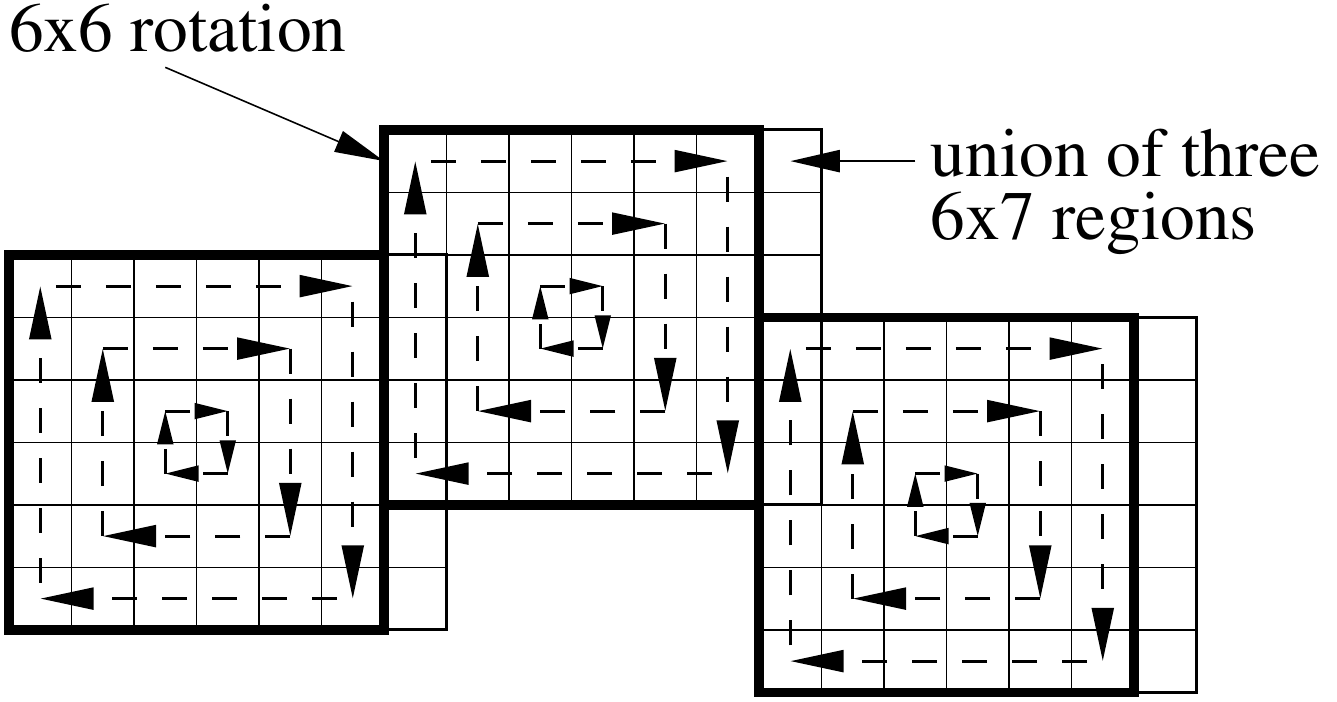}
\caption{The six possible $6\times 6$ block rotations generate all $119!$ permutations of blocks.}\label{fig:floor}
\end{center}
\end{figure}


We use the notation $G=G(g_1, g_2, \dots, g_{q})$ to denote the group generated by all the possible rotations of $k \times k$ squares in the tile arrangement, where the $g_i$ represent the generators, and refer to this as the {\em puzzle group}.  The pattern  of tiles is {\em admissible} if it can be formed by overlapping $k \times (k+1)$ rectangles and/or $(k+1)\times k$ rectangles in a sequence such that if $k$ is even each new rectangle overlaps with the previous arrangement by at least one tile, and if $k$ is odd then it overlaps by at least two adjacent tiles.

When there are $n$ total tiles, the $k \times k$ rotations generate a subgroup of $S_n$. Furthermore, when $k$ is odd there are at least two disjoint orbits: With the tile arrangement colored like a checkerboard, the puzzle tiles permute like-colored squares. (See Figure \ref{fig:orbits}.)

\begin{figure}[ht]
\begin{center}
\includegraphics[height=1.5in]{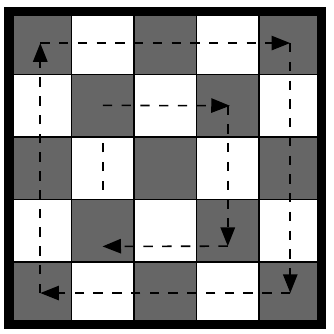}
\caption{The two orbits when $k$ is odd.}\label{fig:orbits}
\end{center}
\end{figure}

Additional restrictions on $G$ are immediately apparent. First, when $k\equiv 0\mod 4$, then each rotation results in an even permutation (of $k^2/4$ four-cycles), and so $G\leq A_n$. Second, when $k$ is odd, then each rotation results in two orbits of $(k^2-1)/8$ four-cycles each.  So assuming the entire union of $k \times (k+1)$ rectangles is colored as a checkerboard, with $m$ squares shaded and $n-m$ squares unshaded, then if $k\equiv 1,7\mod 8$, then $G\leq A_m\times A_{n-m}$, while if $k\equiv 3,5\mod 8$, then $G\leq Even(S_m\times S_{n-m})$, that is, $G$ is a subset of the even permutations in $S_m\times S_{n-m}$. In this paper we prove that in fact these are the only restrictions on the feasible permutations, except in two small cases.

\begin{theorem} \label{thm:main}
If $k>1$ then the puzzle group of an admissible figure on $n$ blocks is given by:

\begin{enumerate}
\item If $k$ is even and $n \neq 6$ then
$$
G = 
\begin{cases}
A_{n} & \textrm{if $k \equiv 0 \mod{4}$} \\
S_{n} & \textrm{if $k \equiv 2 \mod{4}$} 
\end{cases}
$$

\item If $k$ is odd and $n \neq 12$ then shade the figure as a checkerboard in black and white.
If there are $m$ elements in black and $n-m$ in white then
$$
G = 
\begin{cases} 
A_{m} \times A_{n-m}
 &\textrm{if $k \equiv 1,7 \mod{8}$} \\
Even(S_m\times S_{n-m})
 &\textrm{if $k \equiv 3,5 \mod{8}$} 
\end{cases}
$$
where $Even(S_m\times S_{n-m})=(A_{m} \times A_{n-m}) \cup ((S_{m} - A_{m}) \times (S_{n-m} - A_{n-m}))$ is the set of all even permutations in $S_m\times S_{n-m}$.

\item If $n = 6$, then
\[
G= PGL_2(5) \cong S_5
\]
for the projective linear group under an appropriate labeling of vertices.

If $k = 3$ and $n = 12$, then
\[
G\cong S_6
\]
where the projection of $G$ onto each orbit is $S_6$.
\end{enumerate}
\end{theorem}

%
%
%
%
%


\section{The proof}

Our proof is modeled after Wilson's approach to finding the permutation group for a generalization of the $15$--puzzle problem \cite{Wil74.1}.
In order to explain the method we require a few definitions.

Recall that a permutation group $G$ acting on set $X$ is {\em transitive} if it can send any $x$ to any $y$ (i.e. $\forall x,y\in X,\,\exists g\in G:\,gx=y$),
while it is {\em primitive} if it is transitive and does not preserve any bipartition (i.e. $\forall X'\subset X,\,\exists g\in G:\,gX'\not\in\{X',(X')^c\}$).
In particular, a {\em doubly transitive} group (transitive with $stab(x) = \{g\in G:\,gx=x\}$ transitive for some $x$) is primitive.

{\em Jordan's Theorem} says that a primitive group $G$ containing a $3$-cycle is either $A_n$ or $S_n$ (e.g. Theorem 13.3 of \cite{Wie64.1}).
With this in mind, our approach to proving the theorem is to first show that $G$ is doubly transitive on each of its orbits, and then show that $G$ contains a 3-cycle on each orbit (with two exceptions). It follows that $G$ contains the product of the alternating groups on the orbits, which leaves only a small number of potential groups to consider.

%
We first consider the most basic type of puzzle group, that on a $k\times(k+1)$ rectangle.
The general case will be derived from this at the end of the proof.

There are only two generators to consider; denote the generator on the left by $\sigma_\L$, i.e. a clockwise rotation of the left $k\times k$ square region, and the one on the right by $\sigma_\R$, i.e. a clockwise rotation of the right $k\times k$ square region.
The puzzle group is $G=\langle \sigma_\L,\,\sigma_\R\rangle$, the group generated by $\sigma_\L$ and $\sigma_\R$.
Label the tiles of the rectangle by their Cartesian coordinates $(i,j)\in\{1,2,\ldots,k\}\times\{1,2,\ldots,k+1\}$,
with the tile in the upper left corner denoted $(1,1)$, the tile to its right denoted $(1,2)$, and so forth.
Then $\sigma_\L((1,1))=(1,k)$, for instance.
Let ${}_{}^y x=yxy^{-1}$ denote conjugation, so that ${}_{}^y x(a)$ denotes the location of tile $a$ after $y(x(y^{-1}(a)))$.
The conjugate is relatively easy to compute by using the property that if $x(a)=b$ then ${}_{}^y x(y(a))=y(b)$.


\subsection{Building Blocks: Some small products of generators}

It will be useful to note the actions of the generators:
\begin{eqnarray*}
\sigma_\L (i,j)  &=& 
\begin{cases}
(j,k+1-i) & \textrm{for\ \ $j \neq k+1$} \\
(i,k+1)   & \textrm{for\ \ $j = k+1$}
\end{cases} \\
\sigma_\R (i,j) &=& 
\begin{cases}
(j-1,k+2-i) & \textrm{for\ \ $j \neq 1$} \\
(i,1)       & \textrm{for\ \ $j=1$}
\end{cases}
\end{eqnarray*}


The most common approach to proving puzzle groups is to work with a {\em commutator}:
\[
[g,h]=g\,h\,g^{-1}\,h^{-1}
\]
This tends to involve simple shifts from the identity that can be easier to work with than the original action in the puzzle.


We use a few commutators when showing double-transitivity.
\begin{eqnarray*}
[\sigma_\L,\sigma_\R^{-1}](i,j) &=& (i,j-2) \textrm{\qquad when \qquad $i>1$\ \ and\ \ $j>2$} \\
{[}\sigma_\L,\sigma_\R{]}(i,j) &=& (i+2,j)     \textrm{\qquad when \qquad $i<k-1$\ \ and\ \ $1<j<k+1$}
\end{eqnarray*}
We are not concerned with the action outside the specified regions, so we do not describe it here.

For construction of $3$--cycles we find that other expressions which take into account the order of the rotations can be easier to work with.
In this section we develop those building blocks.

The two simplest formulas we can derive are simply rotations by $360^{\circ}$, so $\sigma_\L^4=id$ and $\sigma_\R^4=id$.
More generally, since $\sigma_\L$ and $\sigma_\R$ send a tile $(i,j)$ to nearly the same location,
then a product of four $\sigma_\L$ and $\sigma_\R$ terms will involve only minor shifts for most tiles.
We write out a handful of such expressions and then combine them to get $3$--cycles.


A simple example we will work with is
\[
\sigma_\R^2 \sigma_\L^2 (i,j) = 
\begin{cases}
(i,j+2) & \textrm{if $j < k$} \\
(k+1-i,j-(k-1)) & \textrm{if $j \ge k$}
\end{cases}
\]
Tiles are shifted to the right by $+2$, and when this wraps around the boundary then they are also flipped vertically.
Another useful case is
\[
\sigma_\R^{-1} \sigma_\L (i,j) = 
\begin{cases}
(i+1,j+1) & \textrm{if $i < k$ and $j < k+1$} \\
(j,1)   & \textrm{if $i = k$ and $j < k+1$} \\
(1,i+1) & \textrm{if $j=k+1$}
\end{cases}
\]
Most tiles are shifted down and to the right by one diagonally.

The action of $(\sigma_\R^{-1}\sigma_\L^{-1})^2$ is a bit more complicated:
\[
\left(\sigma_\R^{-1} \sigma_\L^{-1}\right)^2(i,j) = 
\boxed{
\begin{array}{c|ccc|c} 
(k-1,k+1) & (k,k+1) & \hdots  & (3,1)   & (2,1) \\
\hline
(k-2,k+1) & {\mathrm (2,2)   } & {\mathrm \hdots }  & {\mathrm (2,k) }  & (1,1) \\
\vdots    & {\mathrm \vdots    } & {\mathrm \ddots }  & {\mathrm \vdots  }  & \vdots \\
(1,k)   & {\mathrm (k,2)   } & {\mathrm \hdots }  & {\mathrm (k,k) }  & (1,k-1) 
\end{array}
}
\]
The location of $(i,j)$ in the table is the location it is mapped to.
This shows that a cycle formed by the left side, top, and right side
of the rectangle rotates clockwise by $k+1$ tiles, and the rest
remains fixed.


\subsection{Double transitivity}

\begin{lemma} \label{lem:double-transitive}
The puzzle group for a $k\times (k+1)$ rectangle is doubly transitive on each orbit.
\end{lemma}

\begin{proof}
We begin with the case when $k$ is odd.

Let $E$ be the set of tiles with $i+j$ even (the shaded region in Figures \ref{fig:orbits} and \ref{fig:trans-induction}).
This is one of the two orbits of $G$ on the set of tiles.
Since $E^c$ is just the reflection of $E$ through the centerpoint $i\to k+1-i$ and $j\to k+2-j$ then double-transitivity of $E$
also implies double-transitivity on the set of tiles with $i+j$ odd.

Let $a=\left(\frac{k+1}{2},\,\frac{k+1}{2}\right)$ be the square immediately to the left of the center of $\sigma_\R$; this is the center of the $\sigma_\L$ rotation.
Rotations preserve the parity of blocks, i.e. if $(i,j)\in E$ then $\sigma_\L(i,j)\in E$ and $\sigma_\R(i,j)\in E$, so $G\,a = \langle \sigma_\L,\,\sigma_\R \rangle \subseteq E$.

Given $d$ in $\NN$ let $S_0=\{a\}$ and for $d\geq 1$ define
\[
S_d = \left\{(i,j)\in E\setminus\{(1,1)\}:\, \left| i - \frac{k+1}{2}\right| \leq d,\,\left| j - \frac{k+3}{2}\right| \leq d\right\}
\]
This is just those tiles with both $i$ and $j$ coordinates at most $d$ from the center of the $\sigma_\R$ rotation.
We show that $\langle \sigma_\R,\ [\sigma_\L,\sigma_\R^{-1}],\ [\sigma_\L,\sigma_\R]\rangle\, S_{d-1} \supseteq S_d$,
and so by induction $\langle \sigma_\R,\ [\sigma_\L,\sigma_\R^{-1}],\ [\sigma_\L,\sigma_\R]\rangle\,a \supseteq E\setminus \{(1,1)\}$.
Since $\langle \sigma_\R,\ [\sigma_\L,\sigma_\R^{-1}],\ [\sigma_\L,\sigma_\R]\rangle \leq stab(1,1)$ then $stab(1,1)\,a = E\setminus\{(1,1)\}$.

For the base case, when $d=1$ then $a\in S_1$ and $S_1=\bigcup_{\ell=0}^3 \sigma_\R^\ell(a)$, so the claim is trivial.

\begin{figure}[ht] \label{fig:boundaries}
\begin{center}
\includegraphics[height=1.75in]{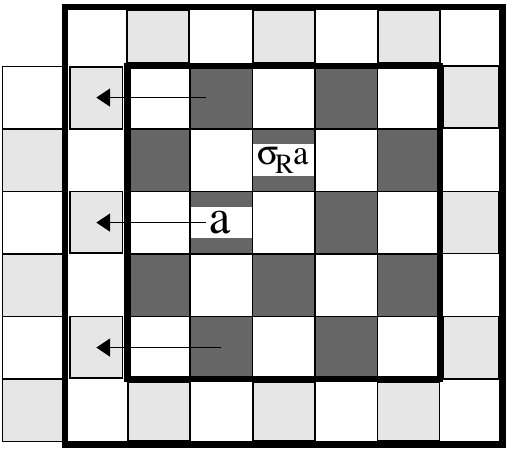}
\caption{$[\sigma_\L,\sigma_\R^{-1}]\,S_2$ contains the left boundary of $S_3$.}\label{fig:trans-induction}
\end{center}
\end{figure}

For the inductive step assume that $d\geq 2$.
If $i\geq 2$ and $j\geq 3$ then $[\sigma_\L,\sigma_\R^{-1}](i,j)=(i,j-2)$,
and so $[\sigma_\L,\sigma_\R^{-1}]\,S_{d-1}$ includes the left boundary of $S_d$.
Rotation of the left boundary under $\langle \sigma_\R\rangle$ includes all of $S_d\setminus S_{d-1}$.
The final case, when $d=\frac{k+1}{2}$, does not require the rotation and completes the proof that $stab(1,1)\,a = E\setminus\{(1,1)\}$.

When $k$ is even let $E$ be the set of all tiles and $a=\left(\frac{k+2}{2},\,\frac{k+2}{2}\right)$.
Then $G\,a\subseteq E$ trivially.
The method of proof is the same, but in the inductive step $[\sigma_\L,\sigma_\R^{-1}]\,S_{d-1}$ now misses the top and bottom of the left boundary of $S_d$ when $d<\frac{k+2}{2}$.
However, $[\sigma_\L,\sigma_\R^{-1}]\,[\sigma_\L,\sigma_\R](i,j)=(i+2,j-2)$ when $i<k-1$ and $j>1$,
and so if $d<\frac{k+2}{2}$ then $\langle \sigma_\R,\ [\sigma_\L,\sigma_\R^{-1}],\ [\sigma_\L,\sigma_\R]\rangle\, S_{d-1}$ also contains the lower left tile of $S_d$.
Applying powers of $\sigma_\R$ to this covers $S_d\setminus S_{d-1}$.
\end{proof}


\subsection{Finding a three-cycle}

Having established double transitivity, we now seek a 3--cycle.

\begin{lemma} \label{lem:3-cycle}
If $k>3$ then there is a 3--cycle in the $k\times (k+1)$ rectangle with generators $\sigma_\L$ and $\sigma_\R$.
More precisely:

\begin{enumerate}
\item If $k$ is even then
\[
\Big(^{\left((\sigma_\R^2\sigma_\L^2)^{\lceil k/4\rceil }\right)} {} \sigma_3\,\left(\sigma_\R^{-1} \sigma_\L^{-1}\right)^2\Big)^{20}
\]
is a $3$--cycle
where $\sigma_3=\sigma_\R$ when $k\equiv 0\mod 4$ and $\sigma_3=\sigma_\L$ when $k\equiv 2\mod 4$.

\item If $k$ is odd then 
\begin{equation} \label{eqn:odd-case}
\left(
(\sigma_\R^{-1} \sigma_\L)^\alpha 
(\sigma_\R^2 \sigma_\L^2)^{(k-1)/2}
\sigma_\L^2
\right)^{\beta/3}
\end{equation}
is a $3$--cycle where
\[
\alpha = 
\begin{cases}
3  & \textrm{if $k\not\equiv 3\mod{18}$} \\
4  & \textrm{if $k\equiv 3\mod{72}$} \\
2  & \textrm{if $k\equiv 21,\, 57\mod{72}$} \\
12 & \textrm{if $k\equiv 39\mod{72}$}
\end{cases}
\]
and $\beta$ is the order of 
$\left(\sigma_\R^{-1} \sigma_\L\right)^\alpha \left(\sigma_\R^2\sigma_\L^2\right)^{(k-1)/2}\sigma_\L^2$.
A 3--cycle on the other orbit may be obtained by swapping $\sigma_\R$ and $\sigma_\L$.
\end{enumerate}

\end{lemma}

\begin{proof}
\noindent {\bf Case 1 ($k$ even): } We consider $k\equiv 0\mod 4$. The methodology when $k\equiv 2\mod 4$ is the same, but with different cycle structures.

The actions of $\left(\sigma_\R^{-1} \sigma_\L^{-1}\right)^2$ and $\sigma_\R^2\sigma_\L^2$ were described earlier.
The exponent in the conjugated term is 
$$
\left(\sigma_\R^2 \sigma_\L^2 \right)^{k/4} (i,j) =
\begin{cases} 
(i,j+\frac k2)         & \textrm{if\ \ $j\le\frac k2+1$} \\
(k+1-i,j-(\frac k2+1)) & \textrm{if\ \ $j>\frac k2+1$} 
\end{cases}
$$
It is a short exercise to verify that $\displaystyle ^{\left((\sigma_\R^2\sigma_\L^2)^{k/4}\right)} {} \sigma_\R$ consists of 
\begin{itemize}
\item $1$-cycles: $(i,\frac k2+1)$ where $1\leq i\leq k$
\item $4$-cycles: $((i,j),\, (\frac k2+j,\frac k2+1+i),\, (i,k+2-j),\, (\frac k2+j,\frac k2+1-i) )$ where $i,\,j \leq \frac k2$
\end{itemize}
The cycle structure of $^{\left((\sigma_\R^2\sigma_\L^2)^{k/4}\right)} {} \sigma_\R \,\left(\sigma_\R^{-1} \sigma_\L^{-1}\right)^2$ then consists of
\begin{itemize}
\item $1$--cycles: $(i,\frac k2+1)$ where $2\leq i \leq k$
\item $3$--cycle: $((1,\frac k2+1),\, (\frac k2 ,k),\, (\frac k2+2,1))$
\item $4$--cycles: One $4$ cycle for each tile
\[
\left\{(i,j)\,:\,2\leq i < \frac k2\ \textrm{and}\ 1\leq j\leq \frac k2\right\} \cup \left\{\left(\frac k2,j\right)\,:\,2 \leq j\leq k-1,\,j\neq \frac k2+1\right\}
\]
\item $10$--cycle: One cycle containing $(k/2,1)$ and $(k/2,k+1)$, among others.
\end{itemize}
This gives $k(k+1)$ tiles, and so it is the complete cycle structure.
Taking the $20^{th}$ power leaves only the square of the $3$--cycle, which is also a $3$--cycle. 

\bigskip\noindent  {\bf Case 2 ($k$ odd) :}
It suffices to show existence of a $3$--cycle on the orbit $E=\{(i,j):\,i+j\ \textrm{even}\}$
since reflecting the $k\times(k+1)$ region through its centerpoint, i.e. $i\to k+1-i$ and $j\to k+2-j$, swaps $E$ and $E^c$ as well as $\sigma_\L$ and $\sigma_\R$, transforming the $3$--cycle on $E$ into a $3$--cycle on $E^c$.

The main term is
\[
\left(\sigma_\R^2 \sigma_\L^2 \right)^{(k-1)/2}(i,j) =
\begin{cases} 
(i,k-1+j)     & \textrm{if\ \ $j=1,2$} \\
(k+1-i,j-2) & \textrm{if\ \ $j>2$} 
\end{cases}
\]
and so
\[
\left(\sigma_\R^2 \sigma_\L^2 \right)^{(k-1)/2}\sigma_\L^2(i,j) =
\begin{cases} 
(i,k-1-j)      & \textrm{if\ \ $j<k-1$} \\
(k+1-i,2k-j) & \textrm{if\ \ $j\geq k-1$} 
\end{cases}
\]


Our theorem uses exponents $\alpha\geq 2$. 
When $\alpha=2$ then \eqref{eqn:odd-case} acts as :
\begin{eqnarray}
\lefteqn{\left(\sigma_\R^{-1} \sigma_\L\right)^2 \left(\sigma_\R^2 \sigma_\L^2 \right)^{(k-1)/2} \sigma_\L^2 (i,j)} \label{eqn:actions-2} \\
&=&
\begin{cases} 
(i+2,k+1-j)
   & \textrm{if\ \ $i \leq k-2$    \ \ and\ \ $j\leq k-2$} \\
(k-j,i-(k-2))
   & \textrm{if\ \ $i >k-2$        \ \ and\ \ $j\leq k-2$} \\
(1,k+1-j)
   & \textrm{if\ \ $(i,j)=(1,k-1)$\ or\ $(1,k)$} \\ 
(k+1-j,k+1-(i-2))
   & \textrm{if\ \ $(i,j)=(i,k-1)$\ or\ $(i,k)$\ with\ $i\geq 2$ } \\
(k,3-i)
   & \textrm{if\ \ $(i,j)=(1,k+1)$\ or\ $(2,k+1)$} \\ 
(k+1-(i-2),k+1)
   & \textrm{if\ \ $(i,j)=(i,k+1)$\ with\ $i>2$}
\end{cases} \nonumber
\end{eqnarray}

When $3\leq\alpha\leq k$ then an inductive argument shows that \eqref{eqn:odd-case} acts as :
\begin{eqnarray}
\lefteqn{\left(\sigma_\R^{-1} \sigma_\L\right)^{\alpha} \left(\sigma_\R^2 \sigma_\L^2 \right)^{(k-1)/2} \sigma_\L^2 (i,j)} \label{eqn:actions} \\
&=&
\begin{cases} 
(\alpha-2-j,i+\alpha)
   & \textrm{if\ \ $i \leq k-\alpha+1$  \ \ and\ \ $j \leq \alpha-3$} \\
(i+\alpha-k-1,\alpha-2-j)
   & \textrm{if\ \ $i>k-\alpha+1$       \ \ and\ \ $j\leq \alpha-3$} \\
(i+\alpha,k+\alpha-1-j)
   & \textrm{if\ \ $i \leq k-\alpha$    \ \ and\ \ $\alpha-3<j<k-1$} \\
(k+\alpha-2-j,i+\alpha-k)
   & \textrm{if\ \ $i >k-\alpha$        \ \ and\ \ $\alpha-3<j<k-1$} \\
(\alpha-i,k+\alpha-1-j)
   & \textrm{if\ \ $i\leq \alpha-1$     \ \ and\ \ $k-1\leq j \leq k+1$} \\ 
(k+\alpha-1-j,k+\alpha+1-i)
   & \textrm{if\ \ $i>\alpha-1$         \ \ and\ \ $k-1\leq j\leq k+1$}
\end{cases} \nonumber
\end{eqnarray}

\noindent{\bf Main Case: $k\not\equiv 3\mod 18$}

Set $\alpha=3$, in accordance with Lemma \ref{lem:3-cycle}.
There is a 3--cycle $\left((1,2),\, (4,k),\, (2,k)\right)$, 
a $10$--cycle containing $(1,1)$ and $9$ other tiles,
an $8$--cycle containing $(2,1)$ and $7$ other tiles,
two $4$--cycles with one containing $(k,1)$ and the other $(k,2)$, 
and a fixed tile $(k-1,2)$.
We now specialize further: 

\bigskip\noindent
{\bf Subcase of $k\equiv 1\mod 6$: } 
Since $\alpha=3$ then the most common transition is when $i\leq k-3$ and $j < k-1$, in which case
\begin{equation} \label{eqn:series}
\onemodsixperm (i,j) = (i+3,k+2-j)
\end{equation}

Starting at $(1,j)$, where $4\leq j\leq k-2$, this induces a long sequence alternating between two types of $j$ terms: 
\begin{equation} \label{eqn:pattern}
(i,j)\to (i+3,k+2-j)\to (i+6,j)
\end{equation}
The pattern doesn't hold if $j<4$ or $j>k-2$, which is why we don't allow either.
In total the sequence contains $\frac{k+2}{3}$ terms, ending in $(k,j)$.

After this comes a short three-term sequence $(k+1-j,3)\to(k+4-j,k-1)\to(3,j)$.
The tile $(3,j)$ satisfies the requirements for \eqref{eqn:series} once again, leading to another sequence using pattern \eqref{eqn:pattern},
this time with $\frac{k-4}{3}$ additional terms, ending in $(k-1,k+2-j)$. 

Following this is another three-term sequence, $(j-1,2)\to(j+2,k)\to(2,k+2-j)$.
Once again repeatedly apply \eqref{eqn:series} to get $\frac{k-4}3$ additional terms,
ending with $(k-2,j)$.

Finally, following this is another three-term sequence, $(k+1-j,1)\to(k+4-j,k+1)\to(1,j)$. 
But $(1,j)$ is just what we started with, and so we are done.

The total number of terms in the cycle is then 
$\frac{k+2}3+3+\frac{k-4}3+3+\frac{k-4}3+2 = k+6$.
There is only one term of the form $(1,j)$ in each sequence, so every $(1,j)$ makes a distinct such sequence, 
and in particular there are $k-5$ such cycles of order $k+6$.

This, along with the $6$ cycles listed before our restriction to $k\equiv 1\mod 6$, accounts for all $k(k+1)$ tiles.
Since $k \equiv 1\mod6$ then $k+6\equiv 1\mod 6$ is not divisible by $3$,
and taking the $(\beta/3)$ power of \eqref{eqn:actions} then leaves only a $3$--cycle.

\bigskip\noindent
{\bf Subcase of $k\equiv 5\mod6$: }
This is nearly identical to the proof when $k\equiv 1\mod 6$, but with sequences ending at slightly different values.
However, again there are $k-5$ cycles of order $k+6$, each containing a member of 
$\{(1,j)\,\mid\,3<j<k-1\}$, and so again there is only one cycle of order divisible by $3$.

\bigskip\noindent
{\bf Subcase of $k \equiv 9,15\mod 18$: } 
Starting at a tile of the form $\left\{(i,j)\,\mid\,1\leq i\leq 3,\,3<j<k-1\right\}$, equation \eqref{eqn:series} gives a sequence of $\frac{k-3}3$ subsequent terms, ending on $(k-3+i,j)$.
This is followed by $(k+1-j,i)\to(k+4-j, k+2-i)\to(i,j)$, for a total of $\frac{k+6}3$ terms in each such cycle.
This gives a family of $3(k-5)$ cycles of order $\frac{k+6}3$ which, when combined with the cycles given before Case 2.1, accounts for all $k(k+1)$ tiles.
Since $\frac{k+6}{3}\equiv 1,\,5 \mod 6$ then $\frac{k+6}{3}$ is not divisible by $3$, and so once again there is only one cycle of order divisible by $3$.

\bigskip\noindent
\noindent{\bf Secondary Case: $k\equiv 3\mod 18$}

The approach used when $k\not\equiv 3\mod 18$ still applies.
However, some of the cycles previously found have order divisible by $3$ when $k\equiv 3\mod 18$, so a different $\alpha$ will be needed.
In fact, numerous subcases with different exponents $\alpha$ are required in order to avoid cycles of order divisible by $3$, and these subcases tend to have many more cycle types.
Following is a chart explaining the cycle structure for these remaining cases.
In each case the cycle type is listed along with exactly one tile from each such cycle.

The simplest case is when $k\equiv 3\mod 72$.
\[
\begin{array}{lllll}
\multicolumn{5}{c}{\underline{k\equiv 3\mod{72}} \ \ (\alpha=4) \hspace{1in}} \\
\underline{\textrm{cycle type}} & \underline{\textrm{one tile}}
 && \underline{\textrm{cycle type}} & \underline{\textrm{one tile}} \\
1 & (k-1,3)
 &\qquad& 10 & (1,2) \\
3 & (1,3) 
 && \frac{k+11}2 & (k-4,1),\, (k-4,2)  \\
4 & (k-2,2),\, (k-2,3) 
 && k+7 & \left\{(1,j)\,\mid\,8\leq j \leq k-2\right\} \\
8 & (2,2)
 && k+11 & (1,1),\, (1,k-1)
\end{array}
\]

When $k\equiv 21\mod 36$ then the theorem uses $\alpha=2$.
The group action in this case is given by \eqref{eqn:actions-2}.
Going through the cycle structure, as in the $\alpha=3$ case, we find it to be:
\[
\begin{array}{lllll}
\multicolumn{5}{c}{\underline{k\equiv 21\mod{36}} \ \ (\alpha=2) \hspace{1in}} \\
\underline{\textrm{cycle type}} & \underline{\textrm{one tile}}
 && \underline{\textrm{cycle type}} & \underline{\textrm{one tile}} \\
1 & (\frac{k+3}2,k+1),\,(k-1,1)
 &\qquad& k+4 & (1,\frac{k+1}{2}) \\
2 & \left\{(i,k+1)\,\mid\,3\leq i \leq \frac{k+1}2\right\}
 && \frac{3k+25}2 & (1,3) \\
3 & (1,1) 
 && 2k+8 & \left\{(1,j)\,\mid\,4\leq j\leq \frac{k-1}2\right\} \\
\frac{k+19}2 & (1,2) 
 && &
\end{array}
\]

\smallskip\smallskip
When $k\equiv 39\mod{72}$ there is a common set of cycle types, plus additional cycles 
depending on the value of $k$ modulo $360$.

\[
\begin{array}{lllll}
\multicolumn{5}{c}{\underline{k\equiv 39\mod{72}} \ \ (\alpha=12) \hspace{1in}} \\
\underline{\textrm{cycle type}} & \underline{\textrm{one tile}}
 && \underline{\textrm{cycle type}} & \underline{\textrm{one tile}} \\
1 & (k-1,11)
 &\qquad& 8 & (2,12) \\
3 & (1,11)
 && 10 & (1,12) \\
4 & (k,11),\, (k,12)
 && \frac{k+9}6 & \{(i,j)\,\mid\,12\leq i\leq k-14,\,1\leq j\leq 3\}
\end{array}
\]
\bigskip
\[
\begin{array}{lllll}
\multicolumn{2}{c}{\underline{k\equiv 39,255\mod{360}}}
  &\qquad& \multicolumn{2}{c}{\underline{k\equiv 111\mod{360}} \hspace{1in}} \\
\underline{\textrm{cycle type}} & \underline{\textrm{one tile}} 
  && \underline{\textrm{cycle type}} & \underline{\textrm{one tile}} \\
\frac16(192+19k+k^2) & (1,1),\, (1,2)
  && \frac{-201+k^2}{120}    & (2,22),\, (2,31) \\
\frac1{12}(111+16k+k^2) & (1,18),\, (1,19)  
  && \frac{309+30k+k^2}{120} & \left\{(i,j)\,\mid\, 1\leq i \leq 8,\, 1\leq j\leq 3\right\} \\
\multicolumn{2}{c}{\underline{k\equiv 183,327\mod{360}}}
  && \frac{-201+k^2}{30}     & (1,21),\, (1,22),\, (1,30),\, (1,31) \\
\frac16(192+19k+k^2) & (1,1),\, (1,2) 
  && \frac{339+20k+k^2}{120} & \{(i,j)\,\mid\,k-8\leq i\leq k-3,\,10\leq j\leq 12\} \\
\frac1{12}(111+16k+k^2) & (1,14),\, (1,15)
  && &  
\end{array}
\]

\end{proof}


\subsection{Proof of Theorem \ref{thm:main}} \label{subsec:proof}

\begin{proof}[Proof of Theorem \ref{thm:main}]
Recall that our proof will utilize the fact that a subgroup of $S_n$ which contains a $3$--cycle and is doubly transitive is either $A_n$ or $S_n$.

The first step is to show that for any admissible figure, $G$ is doubly transitive on its orbit(s). By Lemma \ref{lem:double-transitive} $G$ is doubly transitive on its orbits for a single $k\times (k + 1)$ or $(k+1)\times k$ rectangle. Proceeding by induction, assume that $G$ is doubly transitive on all of its orbits for an admissible figure constructed from $r$ rectangles of dimension $k\times (k + 1)$ or $(k+1)\times k$.
Add another $k\times (k + 1)$ or $(k+1)\times k$ rectangle to this figure so that the resulting figure, which is constructed from $r + 1$ such rectangles, is admissible. There are at least $k$ tiles which belong to the original figure but not to the added rectangle.  
Choose one of these tiles and call it $x$. By hypothesis, the stabilizer of $x$ in the original figure is transitive, while the generators of the new rectangle are transitive.
Since the original figure and the new rectangle overlap in the orbit of $x$, but not at $x$ itself, then the stabilizer of $x$ in the new figure is also transitive and the figure is doubly transitive on the orbit of $x$.
Likewise, if $k$ is odd then $x$ can be chosen to be in either of the two orbits, so the figure is doubly transitive on each orbit. 
Hence for any admissible figure, $G$ is doubly transitive on its orbits.

Next we combine double-transitivity and the $3$--cycles already proven to exist.
This greatly limits the number of groups that are possible, and we then refine this down to a single possible answer in each case.

\bigskip\noindent 
{\bf Case 1 ($k>3$ is even) : } 
Lemma \ref{lem:3-cycle} shows that $G$ contains a 3--cycle.
The previous paragraph establishes that $G$ is doubly transitive, and so $G=A_n$ or $S_n$. 
Since $k$ is even then each generator consists of $\frac{k^2}4$ disjoint $4$--cycles.
As a result, if $k\equiv 0\mod 4$ then the generators are even permutations, and so $G\leq A_n$, implying $G=A_n$.
If $k\equiv 2\mod 4$ then the generators are odd permutations and so $G\neq A_n$, implying $G=S_n$. 

\bigskip\noindent
{\bf Case 2 ($k>3$ is odd) : }
There are two orbits, of some $m$ and $(n-m)$ tiles each, and so $G\leq S_m\times S_{n-m}$. 
From Lemma \ref{lem:3-cycle} there is a $3$--cycle $\sigma\times 1$ on the $m$-element orbit.
The proof of Jordan's Theorem generates $A_m$ by conjugating this specific $3$--cycle and multiplying the resulting terms.
Since ${} ^\tau {}(\sigma\times 1) = {} ^\tau {}\sigma \times 1 \in S_m\times 1$ then Jordan's Theorem implies that $A_m\times 1\leq G$.
Likewise, $1\times A_{n-m}\leq G$. Hence $A_m\times A_{n-m}\leq G$.

Each generator consists of $\frac{k^2-1}4$ disjoint $4$--cycles, exactly $\frac{k^2-1}8$ in each of the two orbits.

If $k\equiv 1,7\mod 8$ then $\frac{k^2-1}8$ is even, and so the $\frac{k^2-1}{8}$ disjoint $4$--cycles in each orbit make an even permutation in that orbit.
It follows that $G\leq A_m\times A_{n-m}$, and so in fact $G=A_m\times A_{n-m}$.

If $k\equiv 3,5\mod 8$ then $\frac{k^2-1}8$ is odd and so the generators $\sigma_\L$ and $\sigma_\R$ act on each orbit as an odd permutation, but are themselves even permutations, and so $A_m\times A_{n-m}\lneq G\leq Even(S_m\times S_{n-m})$.
However, the only group satisfying $A_m\times A_{n-m}\lneq G\leq Even(S_m\times S_{n-m})$ is $G=Even(S_m\times S_{n-m})$.
To see this observe that since $G\leq Even(S_m\times S_{n-m})$ then every $g\in G$ acts as an even permutation on both orbits or as an odd permutation on both orbits.
It follows that if $g\in G\setminus A_m\times A_{n-m}$ and $h\in Even(S_m\times S_{n-m})\setminus (A_m\times A_{n-m})$ then they act as odd permutations on both orbits, and therefore $g\,h^{-1}$ acts as an even permutation on each orbit, i.e. $g\,h^{-1}\in A_m\times A_{n-m}\leq G$ and so $h\in g\,G = G$.

\bigskip\noindent 
{\bf Case 3 ($k=2$) : } 
When $n=6$ (a $2\times 3$ rectangle) then the generators are
\begin{eqnarray}
\sigma_\L &=& ((1,2),\, (2,2),\, (2,1),\, (1,1)) \label{eqn:k2-gen} \\
\sigma_\R &=& ((2,2),\, (1,2),\, (1,3),\, (2,3)) \nonumber
\end{eqnarray}
Equivalently, if we label $(1,2)$ as $\infty$ and then number from $0$ to $4$ counterclockwise starting at $(1,1)\to 0$ and ending at $(1,3)\to 4$, then the generators are $\sigma_\L=(0,\infty,2,1)$ and $\sigma_\R=(\infty,4,3,2)$.
An alternate set of generators is $g_1=\sigma_\L^{-1}=(\infty,0,1,2)$ and $g_2=\sigma_\L^{-1}\sigma_\R^{-1} = (0,1,2,3,4)$.
The projective group $PGL_2(5)$ includes the following transformations on ${\mathbb Z}_5$:
\[
PGL_2(5)=\left\{z\to \frac{az+b}{cz+d}\,:\,a,\,b,\,c,\,d\in {\mathbb Z}_5,\ ad-bc\neq 0\right\}
\]
The generator $g_1$ is the transformation $z\to 3/(z+3)$, while $g_2$ is the transformation $z\to z+1$, so $G\leq PGL_2(5)$.
Those two transformations in fact generate $PGL_2(5)$ (e.g. \cite{Wil74.1}), and so under an appropriate labeling of vertices then $G=PGL_2(5) \cong S_5$.


Suppose instead that $n>6$. The proof of Case 1 carries through as long as there is a $3$--cycle.
The construction of every admissible figure with $n>6$ starts by overlapping two regions of sizes $2\times 3$ and/or $3\times 2$.
Up to symmetry (rotation or reflection through an axis) this region will contain either a $2\times 4$,
a $2\times 3$ joined at a corner square to a $2\times 2$, or two $2\times 3$ joined at a $90^{\circ}$ angle to create a $3\times 3$ missing a corner.
More concretely, let $\sigma_1 = \sigma_\L$ and $\sigma_2 = \sigma_\R$ be the left and right generators defined in \eqref{eqn:k2-gen}.
A third generator $\sigma_3$ and a $3$--cycle will now be designated in each of the three cases just discussed:
\begin{eqnarray*}
\sigma_3 &=& ((1,3),\, (1,4),\, (2,4),\, (2,3)) \\
\textrm{3--cycle} && \left(\sigma_3^2[\sigma_2,\sigma_1]\right)^2 = ((1,2),\, (1,3),\, (2,4)) \\
\sigma_3 &=& ((2,3),\, (2,4),\, (3,4),\, (3,3)) \\
\textrm{3--cycle} && [\sigma_3,\sigma_2] = ((2,2),\,(2,3),\,(2,4)) \\
\sigma_3 &=& ((2,2),\, (2,3),\, (3,3),\, (3,2)) \\
\textrm{3--cycle} && [\sigma_3,\sigma_1] = ((2,1),\, (2,2),\, (2,3))
\end{eqnarray*}

\bigskip\noindent 
{\bf Case 4 ($k=3$) : } 
When $n=12$ then the pair of generators are
\begin{eqnarray}
\sigma_\L &=& 
  ((1,1),\, (1,3),\, (3,3),\, (3,1))\ 
  ((1,2),\, (2,3),\, (3,2),\, (2,1))  \label{eqn:k3-gen} \\  
\sigma_\R &=&
  ((1,2),\, (1,4),\, (3,4),\, (3,2))\ 
  ((1,3),\, (2,4),\, (3,3),\, (2,2)) \nonumber
\end{eqnarray}

Consider action on the orbit $E = \{(i,j):\,i+j\textrm{ is even}\}$.
The puzzle group $G$ is doubly-transitive on each orbit and contains the $3$--cycle  
$\left.(\sigma_\L\,\sigma_\R^2)^2\right|_E = ((1,1),\,(3,1),\,(1,3))$,
and so $A_6 \leq \left.G\right|_E$.
But $\left.G\right|_E$ contains the odd permutation
$\left.\sigma_L\right|_E=((1,1),\, (1,3),\, (3,3),\, (3,1))$,
and so $\left.G\right|_E = S_6$. It can be verified by brute force (e.g. GAP or Mathematica or a very long exercise) that $|G|=6!$, and so in fact $G\cong S_6$.

When $n>12$ then once again start with a $3\times 4$ region and attach a $3\times 4$ or $4\times 3$ to make a larger admissible figure.
This time two $3$--cycles are needed, one in the orbit $E$ and another in the orbit $E^c$.
Up to symmetry (rotation or reflection through an axis) this figure will contain either a $3\times 5$,
a $3\times 4$ joined by two squares near a corner to a $3\times 3$ (two cases),
or two $3\times 4$ joined at a $90^{\circ}$ angle to create a $4\times 4$ missing a corner.
More concretely let $\sigma_1 = \sigma_\L$ and $\sigma_2 = \sigma_\R$ be the left and right generators defined in \eqref{eqn:k3-gen}.
A third generator $\sigma_3$ and a $3$--cycle on each orbit will now be designated in each of the four cases just mentioned.
\begin{eqnarray*}
\sigma_3 &=& ((1,3),\, (1,5),\, (3,5),\, (3,3))\,((1,4),\, (2,5),\, (3,4),\, (2,3)) \\
&&\textrm{3--cycles} \ [\sigma_1,\sigma_3]^2 =((1,4),\,(3,2),\,(2,3)) \\
&&\ \textrm{ and }\ (\sigma_3^2\sigma_2^{-1}\sigma_1)^{20} =((1,3),\, (3,5),\, (2,4))
\\
\sigma_3 &=& ((2,4),\,(2,6),\,(4,6),\,(4,4))\,((2,5),\,(3,6),\,(4,5),\,(3,4)) \\
&&\textrm{3--cycles} \ (\sigma_3^2\,[\sigma_2^2,\sigma_1^2])^2 
\ \textrm{ and }\ (\sigma_1^2\,[\sigma_3^2,\sigma_2^2])^4
\\
\sigma_3 &=&  ((3,3),\,(3,5),\,(5,5),\,(5,3))\,((3,4),\,(4,5),\,(5,4),\,(4,3)) \\
&&\textrm{3--cycles} \ [\sigma_1,\sigma_3]^4 
\ \textrm{ and } \ \left([\sigma_1,\sigma_3]\, [\sigma_2,\sigma_3] \right)^2
\\
%
\sigma_3 &=& ((2,2),\, (2,4),\, (4,4),\, (4,2))\,((2,3),\, (3,4),\, (4,3),\, (3,2)) \\
&&\textrm{3--cycles} \ (\sigma_2\,[\sigma_3^2,\sigma_1^2])^4 
\ \textrm{ and }\ (\sigma_3\sigma_2^2\sigma_1^2)^{20}
\end{eqnarray*}
In some cases the generator $\sigma_3$ may not appear in either of the regions being overlapped, 
such as when overlapping two $3\times 4$ regions to make a $3\times 7$ region. However, we are studying the group 
generated by all the possible rotations of $k \times k$ squares in the tile arrangement, and so $\sigma_3$ is still a 
valid rotation in the union of the two $3\times 4$ regions.
\end{proof}


\section*{Acknowledgments}

We are grateful to the Mathematics Department at Texas Christian University for hosting the REU during which this work was done. Special thanks to George T. Gilbert and Rhonda L. Hatcher for their guidance and to David Addis for his assistance with GAP. Thanks also go to Craig Morgenstern who patiently fixed the disasters we caused to his computer systems.
We particularly appreciate the efforts of an anonymous referee.


\end{document}